\documentclass{amsart}
\usepackage{bbm}
\usepackage{color}
\usepackage{amscd}
\usepackage{nicefrac}
\usepackage{graphicx}
\usepackage{hyperref}
\usepackage{amssymb}
\usepackage{euscript}
\usepackage{enumitem}
\usepackage[cmtip,all]{xy}
\usepackage[export]{adjustbox}

\numberwithin{equation}{section}

\allowdisplaybreaks

\theoremstyle{plain}
\newtheorem{thm}{Theorem}[section]

\newtheorem{theorem}[thm]{Theorem}

\newtheorem{corollary}[thm]{Corollary}

\newtheorem{lemma}[thm]{Lemma}

\newtheorem{definition}[thm]{Definition}

\newtheorem{remark}[thm]{Remark}
\newtheorem{remarks}[thm]{Remarks}

\newtheorem{proposition}[thm]{Proposition}
\newtheorem{example}[thm]{Example}
\newtheorem{non-example}[thm]{Non-example}

\newtheorem{background}[thm]{Background}
\newtheorem{conjecture}[thm]{Conjecture}

\newtheorem*{claim*}{Claim} 
\newtheorem*{lemma*}{Lemma}
\newtheorem*{theorem*}{Theorem}
\newtheorem*{conjecture*}{Conjecture}
\newtheorem{application}[thm]{Application}

\newcommand{\bC}{{\mathbb C}}

\newcommand{\bF}{{\mathbb F}}

\newcommand{\bN}{{\mathbb N}}

\newcommand{\bQ}{{\mathbb Q}}

\newcommand{\bZ}{{\mathbb Z}}

\newcommand{\half}{{\textstyle\frac{1}{2}}}

\newcommand{\quarter}{{\textstyle\frac{1}{4}}}

\newcommand{\iso}{\cong}

\newcommand{\lla}{\langle\hspace{-0.3ex}\langle}
\newcommand{\rra}{\rangle\hspace{-0.3ex}\rangle}

\begin{document}

\title[$P$-adic splittings]{$P$-adic splittings of the quantum connection}
\author{Paul Seidel}

\begin{abstract}
We introduce operations with $p$-adic integer coefficients, associated to idempotents in the quantum cohomology of a monotone symplectic manifold, and apply them to the structure of the quantum connection.
\end{abstract}

\maketitle

\section{Overview\label{sec:overview}}

\subsection{The splitting problem}
Let $M$ be a closed monotone symplectic manifold (for algebraic geometers, complex projective Fano variety). 
Throughout, we will use cohomology with coefficients in an integral domain $R$. Quantum cohomology is $H^*(M)[q]$, where $q$ is a formal variable of degree $2$, equipped with the small quantum product $\ast_q$. Suppose that after inverting $q$, the unit in quantum cohomology can be written as a sum of mutually orthogonal idempotents. This means we have $e_1,\dots,e_m \in H^*(M)[q^{\pm 1}]$ (of degree zero) satisfying 
\begin{equation} \label{eq:idempotents}
1 = e_1 + \cdots + e_m, \quad e_i \ast_q e_j = \begin{cases} e_i & i = j, \\ 0 & i \neq j. \end{cases}
\end{equation}
Quantum product with the $e_i$ yields a splitting of $H^*(M)[q^{\pm 1}]$ into graded $R[q^{\pm 1}]$-modules. For reasons having to do with Fukaya categories \cite{hugtenburg24, chen24}, one expects that there should be induced splittings of other algebraic structures; specifically, we're thinking of the (completed) quantum connection. Namely, introduce another formal variable $t$ of degree $2$, and define the $S^1$-equivariant version of quantum cohomology to be $H^*(M)[q,t]$. The quantum connection is the degree $2$ endomorphism
\begin{equation} \label{eq:quantum-connection}
\nabla_{tq\partial_q} x = tq\partial_q x + c_1(M) \ast_q x.
\end{equation}
We formulate the expectation mentioned above explicitly as:

\begin{conjecture} \label{th:splitting-conjecture}
Given idempotents \eqref{eq:idempotents}, the space $H^*(M)[q^{\pm 1}][[t]]$ carries a canonical splitting into graded $R[q^{\pm 1}][[t]]$-modules, which is invariant under $\nabla_{tq\partial_q}$, and whose $t = 0$ reduction agrees with the splitting of $H^*(M)[q^{\pm 1}]$ given by quantum product with the idempotents. 
\end{conjecture}

It is worth while explaining the notation a little. $H^*(M)[q^{\pm 1}][[t]]$ is the (graded) $t$-completion of $H^*(M)[q^{\pm 1},t]$, so degree $d$ elements are power series
\begin{equation} \label{eq:t-series}
x = \sum_{k=0}^\infty x_k t^k, \quad x_k \in H^*(M)[q^{\pm 1}], \; |x_k| = d-2k.
\end{equation}
As one sees, increasing powers of $t$ are necessarily accompanied by increasingly negative powers of $q$, because of the grading. The use of $t$-completion is natural in terms of the Fukaya-categorical motivation, where it is part of the definition of negative cyclic homology. For an alternative perspective, let's introduce the degree zero variable $\tau = t/q$, writing our space as $H^*(M)[q^{\pm 1}][[\tau]]$. We now focus on a single degree $d$, and write elements as
\begin{equation}
x = \sum_{k=0}^\infty \;\; \sum_{j = \lceil d/2 \rceil - \mathrm{dim}_{\bC}(M)}^{\lfloor d/2 \rfloor} x_{jk} q^j \tau^k, \;\; x_{jk} \in H^{d-2j}(M).
\end{equation}
One has
\begin{equation}
t\partial_q (x_{jk} q^j \tau^k) = x_{jk} (j-k)q^j \tau^{k+1} = \Big(-\tau^2\partial_\tau - \tau \frac{|x|-d}{2}\Big) x_{jk} q^j \tau^k.
\end{equation}
Therefore, the quantum connection $\nabla_{t\partial_q} = q^{-1}\nabla_{tq\partial_q}$ (dividing by $q$ so that it acts on the degree $d$ part) can be equivalently written as
\begin{equation} \label{eq:tau-connection}
\nabla_{-\tau^2\partial_\tau} = -\tau^2 \partial_\tau + (q^{-1} c_1(M) \ast_q \cdot) - \tau \frac{\mathrm{Gr}-d}{2},
\end{equation}
where $\mathrm{Gr}$ multiplies each class in $H^*(M)$ by its degree, and is extended $(t,q)$-linearly. 
When thinking like this in a single degree $d$, one can set $q = 1$ and think of our space as $H^{\mathit{even}}(M)[[\tau]]$ ($d$ even) respectively $H^{\mathit{odd}}(M)[[\tau]]$ ($d$ odd). 
The connection \eqref{eq:tau-connection} has a quadratic pole at $\tau = 0$, corresponding to $q/t = \infty$. The structure of that singularity is involved in some of the major open conjectures in the field \cite{dubrovin98, katzarkov-kontsevich-pantev08, galkin-golyshev-iritani16}. 

\subsection{State of the art}
There are some situations where Conjecture \ref{th:splitting-conjecture} is known to hold, for essentially elementary reasons.

\begin{background} \label{th:eigenvalue-splitting}
Let's require that $H^*(M)$ is a free graded $R$-module, and that the endomorphism $q^{-1} c_1(M) \ast_q: H^*(M)[q^{\pm 1}] \rightarrow H^*(M)[q^{\pm 1}]$ satisfies the following (which can always be achieved by enlarging $R$ appropriately):
\begin{equation} \label{eq:differences-are-invertible}
\parbox{35em}{all eigenvalues lie in $R$; and the difference between any two eigenvalues is invertible.}
\end{equation}
In this situation, for each eigenvalue $\lambda$ there is an idempotent $e_{\lambda} \in H^*(M)[q^{\pm 1}]$ (a polynomial in $q^{-1}c_1(M)$ with respect to the quantum ring structure), which projects to the corresponding generalized eigenspace. The associated splitting of the quantum connection exists and is unique; moreover, any covariantly constant endomorphism must preserve the pieces (Lemma \ref{th:splitting-lemma}). This splitting also has a particularly straightforward categorical interpretation \cite{hugtenburg24}, since the ``Fukaya category of $M$'' is best thought of as a collection of categories indexed by $\lambda$.
\end{background}

\begin{background} \label{th:semisimple} (Communicated to the author by Hugtenburg)
Suppose the quantum cohomology ring is semisimple: there are \eqref{eq:idempotents} which form an $R$-basis, so that
\begin{equation} \label{eq:semisimple}
H^*(M)[q^{\pm 1}] \iso \bigoplus_i R[q^{\pm 1}]e_i \quad \text{as a graded ring.}
\end{equation}
Then, there is a unique splitting of the quantum connection that extends this decomposition. To see that, start with the coarser splitting by eigenvalues from Background \ref{th:eigenvalue-splitting}. Assumption \eqref{eq:semisimple} implies that the $\lambda$-summand of the quantum connection is isomorphic to one of the form
\begin{equation} \label{eq:standard-lambda}
- \tau^2\partial_\tau + \lambda - \tau\frac{\mathrm{dim}_{\bC}(M)-d}{2} + O(\tau^2)
\end{equation}
This is due to Dubrovin \cite[Lecture 3]{dubrovin99} (see \cite[Section 2.4]{galkin-golyshev-iritani16}, \cite[Section 6.1]{hugtenburg24}, or \cite[Lemma 2.1.16]{pomerleano-seidel23} for expositions), and uses the five-point WDVV equation. At this point, assume that $R$ is a field of characteristic $0$. Then \eqref{eq:standard-lambda} has the property that every covariantly constant endomorphism is determined by its $\tau = 0$ part, which can be arbitrary (Lemma \ref{th:non-resonant}). One applies that to the projection matrices given by those $e_i$ that lie in each $\lambda$-eigenspace, and obtains the desired unique splitting.
\end{background}
%
%
%
%
%
%

In spite of these encouraging partial results, Conjecture \ref{th:splitting-conjecture} in general appears to resist an elementary approach: both the existence and uniqueness of splittings with given $\tau = 0$ part are problematic, even for connections with a simple pole (see Examples \ref{th:non-existence} and \ref{th:non-uniqueness}; in Background \ref{th:semisimple}, Dubrovin's result put us in a special situation where those problems does not arise). The Fukaya-categorical approach also runs into a fundamental problem: it relies on knowing that the open-closed map is an isomorphism, which seems hard to establish in general (at least with the current definition of Fukaya category, since it requires one to construct ``enough'' Lagrangian submanifolds; see \cite{ganatra16}). Instead, we draw inspiration from another partial result, which points in a different direction.

\begin{background} \label{th:characteristic-p}
Suppose that $R$ has characteristic $p>0$, meaning that $p = 0 \in R$ for some prime $p$. Then, we have quantum Steenrod endomorphisms \cite{fukaya93b, wilkins18, seidel23, seidel-wilkins21} associated to any class $b \in H^*(M)[q^{\pm 1}]$, which are compatible with the quantum connection. In the special case $b = e_i$, these operations satisfy the same relations as the $e_i$. The resulting splitting works on $H^*(M)[q^{\pm 1},t]$, without formal completion (something which one cannot hope to get in characteristic $0$). 
\end{background}

\subsection{$p$-adic coefficients}
From now on, the standing assumption is:
\begin{equation} \label{eq:p-adic-number-ring}
\parbox{35em}{$R$ is the ring of integers in a $p$-adic number field $K$. Here, $p$ is such that the integral cohomology of $M$ has no $p$-torsion (and hence, the $R$-cohomology is free).}
\end{equation}
Readers unfamiliar with this (as is the author, frankly) may think of the standard $p$-adic numbers, $R = \bZ_p \subset K = \bQ_p$. The general case shares the basic properties of $\bZ_p \subset \bQ_p$: namely, $R$ is a principal ideal domain; it is complete with respect to the decreasing filtration by powers of $p$; and $K$ is obtained from $R$ by inverting $p$.

\begin{remark} \label{th:coefficient-transfer}
Take quantum cohomology with complex coefficients. Any splitting of this is given by idempotents $e_i$ with coefficients in $\bar{\bQ}$, which therefore lie in some number field. Let $K$ be the $p$-adic completion of that number field. If $p$ is large, the coefficients of the $e_i$ will lie in the ring of integers $R \subset K$. In that sense, our framework captures all the splittings that are possible over $\bC$. (This is the reason why we don't just stick to $\bZ_p$.)
\end{remark}

When looking at power series with $R$-coefficients, one can take the $p$-adic filtration into account. Specifically for our case:

\begin{definition} \label{th:log-decay}
Let $R\lla\tau\rra \subset R[[\tau]]$ be the ring of those series 
\begin{equation} \label{eq:tau-series}
x = \sum_{k=0}^\infty x_k \tau^k, \;\; x_k \in R,
\end{equation}
with the following property. There are constants $\alpha,\beta$ such that, for all $m$, the reduction of \eqref{eq:tau-series} modulo $p^m$ is a polynomial in $\tau$ of degree $\leq \alpha p^m  + \beta$. In other words, for all $k> \alpha p^m  + \beta$, the coefficient $x_k$ is divisible by $p^m$.
\end{definition}

This is a slightly sharper condition than $p$-adic convergence on the closed unit disc (the latter amounts to saying that the reduction mod every $p^m$ is a polynomial, without degree bounds; or equivalently, that the coefficients of the series are divisibile by higher and higher powers of $p$).

\begin{remark}
Let's normalize the $p$-adic valuation on $R$ so that it satisfies $\mathrm{val}(p^i) = i$. Then, $x \in R\lla \tau\rra$ is equivalent to saying that there is a constant $\gamma$ such that
\begin{equation}
\mathrm{val}_p(x_k) \geq \log_p(k) - \gamma.
\end{equation} 
Graphically, this means that the Newton polygon of $p$ lies above some vertically shifted version of $y = \log_p(x)$ (see e.g.\ \cite{conrad-series} for an elementary introduction to $p$-adic power series). One can call that logarithmic decay of the coefficients (with slope $1$); this notion has come up previously in the theory of $p$-adic differential equations \cite{dwork-sperber91}.
%
\end{remark}

We define $H^*(M)[q^{\pm 1}]\lla t\rra$ by the same condition. The quantum connection is well-defined on this space, because the operation $tq\partial_q$ preserves $p$-divisibility. Our result is:

\begin{theorem} \label{th:p-adic}
Take $R$ as in \eqref{eq:p-adic-number-ring}. Let $(e_i)$ be a collection of idempotents \eqref{eq:idempotents}. Then there is a canonical splitting of $H^*(M)[q^{\pm 1}]\lla t \rra$, which is invariant under $\nabla_{tq\partial_q}$, and whose $t = 0$ reduction agrees with the splitting given by quantum product with the idempotents.
\end{theorem}

\begin{remark}
Concerning the projection to the $i$-th piece in this splitting, one can be explicit about the constants in Definition \ref{th:log-decay}: $\alpha$ is the $q$-pole order of $e_i$, and $\beta$ is the complex dimension of $M$.
\end{remark}

\begin{application} \label{th:divisibility-for-eigenvalues}
In the situation from Background \ref{th:eigenvalue-splitting}, the splitting is unique, hence Theorem \ref{th:p-adic} provides obtain additional information about it. Let's look at the simplest case, where all eigenvalues are integers. Take $N \in \bN$ such that: the differences of eigenvalues are invertible in $\bZ[1/N]$, and $H^*(M;\bZ[1/N])$ is torsion-free. One can apply Theorem \ref{th:p-adic} for $R = \bZ_p$ for any $p$ coprime to $N$. The outcome is that the projection matrices giving the splitting,
\begin{equation}
E_\lambda = \sum_{k=0}^\infty E_{\lambda}^k\tau^k, \quad E_{\lambda}^k \in \mathit{End}(H^*(M;\bZ[1/N])),
\end{equation}
have the following divisibility property: if $k > \alpha p^m + \beta$ for some $p$ and $m$, then $E_{\lambda}^k$ is divisible by $p^m$ (note that $\alpha$, $\beta$ are independent of $p$).
\end{application}

\begin{application}
The previous observation also applies to the semisimple case from Background \ref{th:semisimple}. Again, suppose for simplicity that the $e_i$ are defined over some ring $\bZ[1/N]$. The elementary argument which constructed the splitting does not control denominators, so a priori that might only be defined over $\bQ$. Theorem \ref{th:p-adic} shows that it is defined over $\bZ[1/N]$, and with the same divisibility property.
\end{application}

\begin{example} \label{th:s2}
Look at $M = \bC P^1$, working first with $\bQ$-coefficients. Take the degree zero part of $H^*(M;\bQ)[q^{\pm 1}][[t]]$, namely $\bQ[[\tau]] \cdot 1 \oplus \bQ[[\tau]] \cdot (q^{-1}h)$, where $h = [\mathit{point}]$. The idempotents associated to the eigenvalues, namely $(1 \pm h)/2$, give rise to splitting of the connection \eqref{eq:tau-connection}. The matrices giving those splittings are $(\mathit{id} \pm H)/2$, where the entries of $H$ are 
\begin{equation}
\begin{aligned}
& H_{21} = 1 + \sum_{j>0} \tau^{2j} \left(\begin{smallmatrix} 2j-1 \\ j \end{smallmatrix}\right)^2 \frac{(2j)!}{2^{8j-2}}, 
\\
& H_{11} = \quarter \tau\partial_\tau(\tau H_{21}), \;\;
H_{22} = -H_{11}, \;\;
H_{12} = H_{21} - \half\tau^2 \partial_\tau(H_{11}), \;\;
\end{aligned}
\end{equation}
The only denominators are powers of $2$, so if we think of these formulae as $p$-adic for $p>2$, then the coefficients are $p$-adic integers ($\half = \frac{1-p}{2}(1+p+p^2+\cdots) \in \bZ_p$). If we reduce mod $p$, the expressions coincide with those in \cite[Example 1.6]{seidel-wilkins21}, which are polynomials in $\tau$ of degree $<p$. More interestingly, the series for $H_{21}$, and hence all the other ones, have $p$-adic radius of convergence $p^{2/(p-1)} >1$, which is stronger than lying in $\bZ_p\lla \tau \rra$. 
\end{example}

\begin{example} \label{th:blowup}
Take $M$ to be the four-torus blown up at a point (this is not monotone, but it's spherically monotone, which is sufficient). We look only at the part of quantum cohomology spanned by: $1$; the class $e$ of the exceptional curve; and that of a point, $-e^2$. The quantum connection (in degree $d = 2$) is
\begin{equation} 
-\tau^2 \partial_{\tau} + \begin{pmatrix} \tau & 0 & 0 \\ -1 & -1 & 0 \\ 0 & 1 & -\tau \end{pmatrix}.
\end{equation} 
The splitting for the eigenvalue $1$ of quantum multiplication (equivalently, the covariantly constant extension of quantum multiplication with $e$) is given by the idempotent matrix $E$ with nonzero entries
\begin{equation}
E_{22} = 1, \;\;
E_{12} = \sum_{j\geq 0} (-1)^j j! \tau^j, \;\;
E_{23} = -\sum_{j \geq 0} j! \tau^j, \;\;
E_{13} = E_{12}E_{23}.
\end{equation}
The $p$-adic radius of convergence is $p^{1/(p-1)} > 1$.
\end{example}

We do not know whether the overconvergence phenomenon observed in these examples applies more generally to quantum connections. 
%
Finally, we can give a half-answer to Conjecture \ref{th:splitting-conjecture} in the classical context:

\begin{application}
Take the maximal decomposition of quantum cohomology over $\bar{\bQ}$ (the unique one with the largest number of idempotents $e_i$). Following Remark \ref{th:coefficient-transfer}, those idempotents give rise to ones defined over the ring of integers in some $p$-adic number field $K$. Take the splitting provided by Theorem \ref{th:p-adic}. Via the embedding of $\bar{\bQ}$ into the algebraic closure of $K$, the existence of a splitting over $K[[\tau]]$ implies that of one over $\bar{\bQ}[[\tau]]$ (Proposition \ref{th:constructible-2}). Because of the maximality assumption, the $\tau = 0$ reduction of any such $\bar\bQ[[\tau]]$-splitting must reproduce that given by $(e_i)$. This is a pure existence results, which fails to address the ``canonical'' part of Conjecture \ref{th:splitting-conjecture}.
%
\end{application}

\subsection{Idea of the construction}
Theorem \ref{th:p-adic} is an extension of Background \ref{th:characteristic-p}. Let's recall the definition of the quantum Steenrod endomorphisms, using coefficients in $\bF_p$ for the sake of familiarity. Fix a class $b \in H^*(M;\bF_p)[q^{\pm 1}]$ of even degree (the last-mentioned assumption is in principle unnecessary, but all our applications will satisfy it). The associated operation is a map of degree $p|b|$,
\begin{equation} \label{eq:quantum-steenrod-1}
Q\Sigma_{1,b}: H^*(M;\bF_p)[q^{\pm 1}] \longrightarrow H^*(M;\bF_p)[q^{\pm 1},t,\theta_1]
\end{equation}
where $\theta_1$ has degree $1$ (it is an odd variable, so $\theta_1^2 = 0$ if $p>2$; and $\theta_1^2 = t$ for $p = 2$). 
The definition is based on pseudo-holomorphic spheres in $M$, carrying $(p+2)$ marked points, which satisfy the intersection constraints from Figure \ref{fig:bbb}.
\begin{figure}
\begin{centering}
\includegraphics{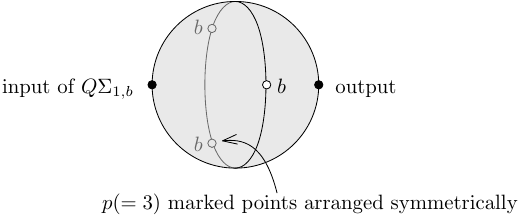}
\caption{\label{fig:bbb}Picture of \eqref{eq:quantum-steenrod-1}.}
\end{centering}
\end{figure}%
That picture has an obvious symmetry group $\Gamma_1 = \bZ/p$, and one works equivariantly (in the sense of the Borel construction) with respect to that symmetry. The operations are correspondingly indexed by $H^*(B\Gamma_1;\bF_p) = \bF_p[t,\theta_1]$.

Let's increase the number of extra marked points to $p^2$, and the symmetry to $\Gamma_2 = \bZ/p^2$. The coefficient ring is similar, $H^*(B\Gamma_2;\bF_p) = \bF_p[t,\theta_2]$, and the operations are correspondingly maps of degree $p^2|b|$,
\begin{equation}
Q\Sigma_{2,b}: H^*(M;\bF_p)[q^{\pm 1}] \longrightarrow H^*(M;\bF_p)[q^{\pm 1},t,\theta_2].
\end{equation}
Restriction to the subgroup $\Gamma_1 \subset \Gamma_2$ gives rise to a map
\begin{equation} \label{eq:gamma12}
H^*(B\Gamma_2;\bF_p) \longrightarrow H^*(B\Gamma_1;\bF_p), \quad
t \mapsto t, \; \theta_2 \mapsto 0.
\end{equation}
What is the relation with the standard quantum Steenrod operations? If one takes the surface underlying $Q\Sigma_{2,b}$ and reduces the symmetry to $\Gamma_1 \subset \Gamma_2$, there is some freedom to move the $p^2$ points around. In particular one can combine them into groups of $p$, which bubble off as follows:
\begin{equation}
\includegraphics{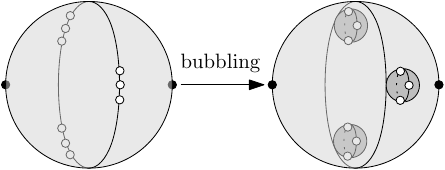}
\end{equation}
The outcome is this commutative diagram, where $b^{\ast_q p}$ is the $p$-fold quantum power of $b$:
\begin{equation} \label{eq:b1-b2}
\xymatrix{
H^*(M;\bF_p)[q^{\pm 1}] \ar[rr]^-{Q\Sigma_{1,b^{\ast_q p}}} && 
H^*(M;\bF_p)[q^{\pm 1},t,\theta_1]
\\
H^*(M;\bF_p)[q^{\pm 1}] \ar[rr]^-{Q\Sigma_{2,b}} \ar@{=}[u] && 
H^*(M;\bF_p)[q^{\pm 1},t,\theta_2] \ar[u]_-{\eqref{eq:gamma12}}
}
\end{equation}

Suppose that instead of coefficients in $\bF_p$, we use ones in $\bZ_p$. For the $\Gamma_1$-equivariant operations, the nontrivial part of the coefficient ring is still annihilated by $p$:
\begin{equation} \label{eq:bgamma-1}
H^*(B\Gamma_1;\bZ_p) = \bZ_p \oplus t\bF_p \oplus t^2\bF_p \oplus \cdots
\end{equation}
where $t$ maps to the corresponding element under reduction to $\bF_p$-coefficients. This explains why the operations \eqref{eq:quantum-steenrod-1} rarely (in fact, never under our assumptions \eqref{eq:p-adic-number-ring}, because all the cohomology can be lifted to $\bZ_p$) have a nontrivial $\theta_1$-component; other than that, it provides nothing new. On the next level,
\begin{equation} \label{eq:bgamma-2}
H^*(B\Gamma_2;\bZ_p) = \bZ_p \oplus t(\bZ/p^2) \oplus t^2(\bZ/p^2) \oplus \cdots,
\end{equation}
with the map from \eqref{eq:bgamma-2} to \eqref{eq:bgamma-1} being the obvious reduction mod $p$. The analogue of \eqref{eq:b1-b2} shows that $Q\Sigma_{1,b^{\ast_q p}}$ admits a mod $p^2$ lift, which is given by $Q\Sigma_{2,b}$. This in particular applies to idempotent elements $b$, which are their own $p$-th powers. For such idempotents, one can consider the tower of operations $Q\Sigma_{m,b}$ with symmetry $\Gamma_m = \bZ/p^m$ for any $m$, where the relevant coefficient ring is
\begin{equation} \label{eq:bgamma-3}
H^*(B\Gamma_m;\bZ_p) = \bZ_p \oplus t(\bZ/p^m) \oplus t^2(\bZ/p^m) \oplus \cdots
\end{equation}
In the (inverse) limit, one obtains operations indexed by 
\begin{equation}
H^*(B\Gamma_\infty;\bZ_p) = \bZ_p[t],
\end{equation}
where $\Gamma_\infty$ is the union of all $\Gamma_1 \subset \Gamma_2 \subset \cdots$. That is how the splittings in Theorem \ref{th:p-adic} are constructed. As we increase $m$, the number of incidence conditions with $b$ grows: hence, so does the inverse power of $q$, at a rate that is at most a constant times $p^m$. This, and degree considerations, leads to the occurrence of rings $\bZ_p\lla \tau \rra$. The geometric side of the construction  is essentially the same as in the $m = 1$ case from \cite{seidel-wilkins21}, with the difference being that the resulting information is inserted into a more refined algebraic setup.

\begin{remarks}
(i) One could think of a more general case where the idempotents are defined over a $p$-adic field $K$, but our construction does not adapt to that situation. Namely, if $b$ is such that $p^d b$ has coefficients in the ring of integers $R$, for some $d>0$, one wants to set ``\,$Q\Sigma_{m,b} = p^{-d \, p^m} Q\Sigma_{m,p^d b}$\!''; but that makes no sense, given that the nontrivial part of \eqref{eq:bgamma-3} is $p^m$-torsion. 

(ii) The quantum connection belongs to an $S^1$-equivariant world. The idea of using the discrete group $\Gamma_\infty = \bZ/p^\infty$ as a replacement for the topological group $S^1$ is by no means new (for an instance on a much deeper level than here, see \cite[Section II.1]{nikolaus-scholze18}). Of course, there are many other constructions in symplectic topology where one does work $S^1$-equivariantly; one that's close to our situation is the ``cap product'' action of the cohomology of the loop space on symplectic cohomology \cite{viterbo95, viterbo97a}.

(iii) More generally, but still by the same means, one could define the operations from Theorem \ref{th:p-adic}  in the ``weakly monotone'' situation. Beyond monotonicity, this means requiring that $\mathrm{dim}_{\bC}(M) \leq 3$, or else that $c_1(M)$ is divisible by $\mathrm{dim}_{\bC}(M)-2$. Unfortunately, those other cases do not mesh well with the existence of interesting idempotents in quantum cohomology: for instance, blowups of complex codimension two submanifolds would qualify only for $\mathrm{dim}_{\bC}(M) \leq 3$. 
\end{remarks}

\section{Equivariant cohomology}

\subsection{Group cohomology}
Let $\Gamma_m = \bZ/p^m$ be the finite cyclic group of order $p^m$, with generator $\sigma_m$. To compute the cohomology of that group with $R$-coefficients, where $R$ is as in \eqref{eq:p-adic-number-ring}, one can use the following two-periodic free resolution of the trivial $R[\Gamma_m]$-module:
\begin{equation} \label{eq:free-resolution}
C_*(E\Gamma_m) = \left\{
R[\Gamma_m] \xleftarrow{\sigma_m-1} R[\Gamma_m] \xleftarrow{1+\sigma_m+\cdots+\sigma_m^{p^m-1}} R[\Gamma_m] \xleftarrow{\sigma_m-1} \cdots \right\}
\end{equation}
From that, one gets the group cochain complex
\begin{equation} \label{eq:standard-complex}
C^*(B\Gamma_m) = 
\mathit{Hom}_{R[\Gamma_m]}(C_*(E\Gamma_m),R) = 
\left\{ R \xrightarrow{0} R \xrightarrow{p^m} R \xrightarrow{0} R \xrightarrow{p^m} \cdots \right\}
\end{equation}
The group cohomology is accordingly
\begin{equation}
H^*(B\Gamma_m) = \begin{cases} R & \ast = 0, \\
R/p^m & \ast > 0 \text{ even,} \\
0 & \ast \text{ odd.}
\end{cases}
\end{equation}
(We are already using topological notation, even if the constructions are set up in purely algebraic terms.) To get the cup product on group cohomology, one needs a diagonal map for the resolutions (see e.g.\ \cite[Chapter 5]{brown82}). We use
\begin{equation} \label{eq:diagonal-map}
\begin{aligned}
& C_*(E\Gamma_m) \longrightarrow C_*(E\Gamma_m) \otimes_R C_*(E\Gamma_m), \\
& c_i \longmapsto \sum_{j+k=i} c_j \otimes c_k + \sum_{\substack{j+k=i-1 \\ 0 \leq r<s < p^m}} \sigma_m^r( c_j\gamma) \otimes \sigma_m^s(c_k\gamma), \\
& c_i \gamma \longmapsto \sum_{j+k=i} c_j \otimes c_k\gamma + c_j\gamma \otimes \sigma_m(c_k).
\end{aligned}
\end{equation}
Here, $c_i$ stands for the obvious generator of \eqref{eq:free-resolution} in degree $-2i$ (all our complexes are cohomologically graded), and $c_i\gamma$ for the same in degree $-2i-1$. The map \eqref{eq:diagonal-map} is $R[\Gamma_m]$-linear, where the action on the right hand side is the diagonal one. From that, one gets the product
\begin{equation} \label{eq:cup-product}
\begin{aligned}
& C^*(B\Gamma_m) \otimes C^*(B\Gamma_m) \longrightarrow C^*(B\Gamma_m), \\
& t^j \otimes t^k \longmapsto t^{j+k}, \\
& t^j\theta \otimes t^k \longmapsto t^{j+k}\theta, \\
& t^j \otimes t^k\theta \longmapsto t^{j+k}\theta, \\
& t^j\theta \otimes t^k\theta \longmapsto \textstyle\frac{p^m(p^m-1)}{2}  t^{j+k+1}.
\end{aligned}
\end{equation}
Here, $t^i$ is the obvious generator of \eqref{eq:standard-complex} in degree $2i$, and $t^i\theta$ that in degree $2i+1$. On $H^*(B\Gamma_m)$, this yields a polynomial algebra structure with $t = t^1$ as generator.

\begin{remark}
Let's temporarily switch to mod $p^m$ coefficients, where $H^*(B\Gamma_m;R/p^m) = R/p^m$ in each nonnegative degree. Then, the last line of \eqref{eq:cup-product} reduces to
\begin{equation}
\theta \otimes \theta \longmapsto \begin{cases} 2^{m-1} t & p = 2, \\
0 & p>2.
\end{cases}
\end{equation}
This matches what one knows to be true from topology (that $\theta^2$ has to be $2$-torsion).
\end{remark}

Multiplication by $p$ yields an inclusion $\Gamma_m \hookrightarrow \Gamma_{m+1}$, $\sigma_m \mapsto \sigma_{m+1}^p$. On the resolutions \eqref{eq:free-resolution}, one has corresponding maps 
\begin{equation} \label{eq:map-of-resolutions}
\xymatrix{
R[\Gamma_m] \ar[d]^-{1} 
&&
R[\Gamma_m] \ar[ll]_-{\sigma_m-1} \ar[d]^-{1+\sigma_{m+1}+\cdots+\sigma_{m+1}^{p-1}}
&&
R[\Gamma_m] \ar[ll]_-{1+\sigma_m+\cdots} \ar[d]^-{1} & \ar[l] \cdots
\\
R[\Gamma_{m+1}] 
&&
\ar[ll]_-{\sigma_{m+1}-1} 
R[\Gamma_{m+1}]
&&\
\ar[ll]_-{1+\sigma_{m+1}+\cdots}
R[\Gamma_{m+1}] & \ar[l] \cdots
}
\end{equation}
These are maps of $R[\Gamma_m]$-modules, where the module structure on the bottom row is that induced the inclusion of groups. The labels on the vertical arrows indicate the image of the generator $1$. For instance, the middle $\downarrow$ is
\begin{equation}
\sigma_m^k \longmapsto \sigma_{m+1}^{pk}(1 + \sigma_{m-1} + \cdots + \sigma_{m+1}^{p-1}) =
\sigma_{m+1}^{pk} + \sigma_{m+1}^{pk+1} + \cdots +
\sigma_{m+1}^{pk+(p-1)}.
\end{equation}
The induced map on \eqref{eq:standard-complex} is 
\begin{equation} \label{eq:induced-map}
\xymatrix{
R \ar[r]^0 & 
R \ar[r]^p &
R \ar[r] & \cdots
& 
\\
R \ar[r]^0 \ar[u]_1 & 
R \ar[r]^{p^2} \ar[u]_p &
R \ar[u]_1 \ar[r] & \cdots
}
\end{equation}
On cohomology, one gets that 
\begin{equation} \label{eq:restrict-group}
H^*(B\Gamma_{m+1}) \longrightarrow H^*(B\Gamma_m)
\end{equation}
is the quotient map $R/p^{m+1} \rightarrow R/p^m$ in all nontrivial degrees. Let $\Gamma_\infty$ be the union (direct limit) of the $\Gamma_m$; in other words, the discrete subgroup of $S^1$ consisting of elements whose order is a power of $p$. 

\begin{lemma} \label{th:gamma-infinity}
Restriction to $\Gamma_m$ induces an isomorphism
\begin{equation} \label{eq:limit-cohomology}
H^*(B\Gamma_\infty) \iso \underleftarrow{\lim}_m \; H^*(B\Gamma_m) = \begin{cases}
R & \ast \text{ even,} \\
0 & \ast \text{ odd.}
\end{cases}
\end{equation}
\end{lemma}

\begin{proof}
The corresponding statement for group homology is standard: 
\begin{equation}
H_*(B\Gamma_\infty) \iso \underrightarrow{\lim}_m \;H_*(B\Gamma_m). 
\end{equation}
Dualizing shows that $H^*(B\Gamma_\infty)$ is isomorphic to the cohomology of the derived inverse limit of $C^*(B\Gamma_m)$. Now, inverse limits are exact for finitely generated $R$-modules \cite[Theorem 1]{jensen70}. Hence, one can equivalently use the naive inverse limit of the $C^*(B\Gamma_m)$ (even though it violates the Mittag-Leffler condition). That naive limit is
\begin{equation}
R \longrightarrow 0 \longrightarrow R \longrightarrow 0 \longrightarrow R \rightarrow \cdots
\end{equation}
and its cohomology agrees with $\underleftarrow{\lim}_m \; H^*(B\Gamma_m)$.
\end{proof}

\begin{lemma}
The inclusion $\Gamma_\infty \rightarrow S^1$ (considering $\Gamma_\infty$ as a discrete group, and $S^1$ as a topological group) induces an isomorphism $H^*(BS^1) \iso H^*(B\Gamma_m)$.
\end{lemma}

\begin{proof}
This is best seen topologically: $B\Gamma_m$ is a circle bundle over $BS^1$, so one has a Gysin sequence
\begin{equation} \label{eq:gysin}
\cdots \rightarrow H^*(BS^1) \longrightarrow H^*(B\Gamma_m) \rightarrow H^{*-1}(BS^1) \rightarrow \cdots
\end{equation}
This shows that $H^*(BS^1) \rightarrow H^*(B\Gamma_m)$ is onto (the identity in degree $0$, and reduction mod $p^m$ in positive even degrees). Passing to the inverse limit yields the desired result.
\end{proof}


\subsection{Equivariant cohomology with coefficients}
Let $V$ be a chain complex of $R$-modules, with an action of $\Gamma_m$. Define
\begin{equation} \label{eq:v-coefficients}
C^*(B\Gamma_m;V) = \mathit{Hom}_{R[\Gamma_m]}(C_*(E\Gamma),V) = \left\{ V \xrightarrow{\sigma_m - 1} V \xrightarrow{1+\sigma_m + \cdots} V \xrightarrow{\sigma_m - 1} \cdots \right\}
\end{equation}
Here, the notation stands for collapsing a bicomplex (with the standard Koszul signs), but taking the direct product of the copies of $V$ involves. Write $H^*(B\Gamma_m;V)$ for the cohomology. This construction is obviously functorial under (homotopy classes of) $\Gamma_m$-equivariant chain maps. The relevant generalization of \eqref{eq:induced-map}, derived as before from \eqref{eq:map-of-resolutions}, is
\begin{equation} \label{eq:induced-map-with-coefficients}
\begin{aligned}
& C^*(B\Gamma_{m+1};V) \longrightarrow C^*(B\Gamma_m;V), \\
& t^i v \longmapsto t^i v, \\
& t^i\theta v \longmapsto t^i\theta (v + \sigma_{m+1} v + \cdots + \sigma_{m+1}^{p-1} v).
\end{aligned}
\end{equation}
There is also a generalization of the cup product,
\begin{equation} \label{eq:cup-product-with-coefficients}
H^*(B\Gamma_m;V) \otimes H^*(B\Gamma_m;W) \longrightarrow H^*(B\Gamma_m;V \otimes W),
\end{equation}
where $V \otimes W$ carries the diagonal action. One can use \eqref{eq:diagonal-map} to derive an explicit chain level formula for this product, generalizing \eqref{eq:cup-product}:
\begin{equation} \label{eq:chain-cup-product-with-coefficients}
\begin{aligned}
& C^*(B\Gamma_m;V) \otimes C^*(B\Gamma_m;W) \longrightarrow C^*(B\Gamma_m;V \otimes W), \\
& t^i v \otimes t^j w \longmapsto t^{i+j} (v \otimes w), \\
& t^j\theta v \otimes t^k w \longmapsto t^{j+k}\theta (v \otimes \sigma_m(w)), \\
& t^j v \otimes t^k\theta w\longmapsto t^{j+k}\theta (v \otimes w), \\
& t^j\theta v \otimes t^k\theta w \longmapsto \!\!\!\sum_{0 \leq r < s < p^m} t^{j+k+1}\big( \sigma_m^r(v) \otimes \sigma_m^s(w) \big).
\end{aligned}
\end{equation}

For any chain complex $B$, we can consider $V = B^{\otimes p^m}$ with the $\Gamma_m$-action which cyclically permutes the factors (with signs). If $b \in B$ is a cocycle of even degree, then $b^{\otimes p^m} \in C^0(B\Gamma_m;B^{\otimes p^m})$ is an equivariant cocycle. This construction satisfies (see e.g.\ \cite[Lemma 2.5]{seidel-wilkins21}):

\begin{lemma} \label{th:diagonal}
The class of $[b^{\otimes p^m}] \in H^*(B\Gamma_m;B^{\otimes p^m})$ depends only on $[b] \in H^*(B)$. 
\end{lemma}


In the special case where $V = C^*(X)$ is the ($R$-coefficient, cellular or singular) chain complex of a space with a $G$-action, $H^*(B\Gamma_m;C^*(X)) = H^*_{\Gamma_m}(X)$ is (Borel) equivariant cohomology. We will need to look at one instance of this, where the space is the two-sphere
\begin{equation}
S = \bar{\bC} = \bC \cup \{\infty\},
\end{equation}
with the rotational action of $\Gamma_m$. We have integration over the equivariant fundamental cycle, as well as restriction to the fixed points:
\begin{equation} \label{eq:cohomology-of-s} 
\begin{aligned}
& {\textstyle \int_S}: H^*_{\Gamma_m}(S) \longrightarrow
H^{*-2}_{\Gamma_m}(\mathit{point}), \\
& \rho_0,\rho_\infty : H^*_{\Gamma_m}(S) \longrightarrow H^*_{\Gamma_m}(\mathit{point}), 
\end{aligned}
\end{equation}

\begin{lemma} \label{th:localisation}
The following diagram commutes:
\begin{equation}
\xymatrix{
\ar@/_1pc/[rrrr]_-{\rho_0 - \rho_\infty}
H^*_{\Gamma_m}(S) \ar[rr]^-{\int_S} &&
H^{*-2}_{\Gamma_m}(\mathit{point}) \ar[rr]^-{t} && H^*_{\Gamma_m}(\mathit{point})
}
\end{equation}
\end{lemma}

\begin{proof}
Generalizing \eqref{eq:gysin}, one has that $X \times_{\Gamma_m} E\Gamma_m$ is a circle bundle over $X \times_{\Gamma_m} ES^1$, leading to a Gysin sequence
\begin{equation}
\cdots \rightarrow
H^*_{S^1}(X) \longrightarrow H^*_{\Gamma_m}(X) \rightarrow H^{*-1}_{S^1}(X) \rightarrow \cdots
\end{equation}
Take $X = S$. Since $H^*_{S^1}(S)$ is concentrated in even degrees, the map to $\Gamma_m$-equivariant cohomology is onto, hence it suffices to prove the statement for $H^*_{S^1}$ instead, but that is an instance of the standard localisation theorem.
\end{proof}

\section{The $p$-adic operation}

\subsection{Construction for fixed $m$}
The geometric construction generalizes \cite[Section 4]{seidel-wilkins21} in a straightforward way. We will summarize the strategy and outcome, with just enough details so that one check that the generalization goes through, and then explain the algebraic formalism into which the outcome is inserted. Choose $m \geq 1$ and $A \in H_2(M;\bZ)$. Everything is based on parametrized moduli spaces of pseudo-holomorphic maps, with Morse-theoretic adjacency conditions. These are spaces of pairs $(w,u)$ of the following kind:
\begin{equation} \label{eq:parametrized-maps}
\left\{
\begin{aligned}
& w \in S^\infty, \\
& u: S = \bar{\bC} \longrightarrow M, \;\; [u] = A, \\
& \half(du + J \circ du \circ j)_z = \nu_{w,z,u(z)}: TS_z \rightarrow TM_{u(z)}, \\
& u(0) \in W^u(x_0), \; u(\infty) \in W^s(x_\infty), \\
& u(\zeta_{m,j}) \in W^u(x_j) \text{ for $j = 1,\dots,p^m$, where $\zeta_{m,j} = e^{2\pi i (j-1/2)/p^m}$.}
\end{aligned}
\right.
\end{equation}
Here, $S^\infty \subset \bC^\infty$ is our model for $E\Gamma_m$, with the action rotating all complex coordinates. $J$ is a compatible almost complex structure on $M$. The inhomogeneous term $\nu$ is parametrized by $w \in S^\infty$, and satisfies a $\Gamma_m$-equivariance condition with respect to the action on $S^\infty \times S$ \cite[Equation (4.1)]{seidel-wilkins21}.  We fix a Morse function $f$ and Riemannian metric on $M$, forming a Morse-Smale pair. $(x_0,x_1,\dots,x_{p^{m}},x_\infty)$ are critical points of $f$, and $W^s/W^u$ their stable/unstable manifolds for the gradient flow (one can picture the incidence conditions as a pseudo-holomorphic sphere with Morse half-trajectories attached, as in Figure \ref{fig:half-trajectories}). We use the cell decomposition of $S^\infty$ from \cite[Section 2a]{seidel-wilkins21}, where the cells of each dimension form a free $\Gamma_m$-orbit
\begin{equation} \label{eq:cells-1}
\Delta_{m,d},\, \sigma_m(\Delta_{m,d}),\, \dots,\, \sigma_m^{p^m-1}(\Delta_{m,d}) \subset S^\infty.
\end{equation}
Explicitly, in coordinates $(w_0,w_1,\dots) \in S^\infty \subset \bC^\infty$, the defining equations for $\Delta_{m,d}$ are
\begin{equation}
\label{eq:cells-2}
\begin{cases} 
w_{d/2} \geq 0, \; w_{d/2+1} = w_{d/2+1} = \cdots = 0 & \text{if $d$ is even}, \\
e^{-2\pi i \theta} w_{d/2-1/2} \geq 0 \text{ for $\theta \in [0,2\pi/p^m]$,} \;
w_{d/2+1/2} = w_{d/2+3/2} = \cdots = 0 & \text{if $d$ is odd.} 
\end{cases}
\end{equation}
With suitable orientations, the cellular chain complex reproduces \eqref{eq:free-resolution} \cite[Equations (2.7), (2.8)]{seidel-wilkins21}:
\begin{equation} \label{eq:cellular-differential}
\partial [\Delta_{m,d}] = 
\begin{cases} 
[\Delta_{m,d-1}] + \sigma_m[\Delta_{m,d-1}] + \cdots + \sigma_m^{p^m-1}[\Delta_{m,d-1}]
& \text{$d>0$ even,} \\ 
\sigma_m[\Delta_{m,d-1}] - [\Delta_{m,d-1}]
& \text{$d$ odd.}
\end{cases}
\end{equation}
One restricts \eqref{eq:parametrized-maps} to $w \in \Delta_{m,d} \setminus \partial \Delta_{m,d}$. Counting points in the resulting zero-dimensional moduli spaces yields numbers 
\begin{equation} \label{eq:counting}
\begin{aligned}
& n_A(\Delta_{m,d},x_0,x_1,\dots,x_{p^m},x_\infty) \in \bZ
\\ & \qquad \text{ for } |x_0|-|x_1|-\cdots-|x_{p^m}|-|x_\infty| + 2\textstyle\int_A c_1(M) + d = 0.
\end{aligned}
\end{equation}
Write $n(x_-,x_+) \in \bZ$ for the count of gradient flow lines connecting orbits with $|x_+| =|x_-| + 1$, which defines the Morse cohomology differential. An analysis of the ends of the one-dimensional spaces \cite[Lemma 4.1]{seidel-wilkins21} shows that \eqref{eq:counting} satisfy the following relations:
\begin{equation} \label{eq:parametrized-numbers}
\begin{aligned}
& 
\sum_x \pm n(x_0,x) n_A(\Delta_{m,d},x,x_1,\dots,x_{p^m},x_\infty) \\ & \quad +
\sum_{x,k} \pm n(x_k,x) n_A(\Delta_{m,d},x_0,x_1,\dots,x_{k-1} x,x_{k+1},\dots,x_\infty) \\ & \quad +
\sum_x \pm n_A(\Delta_{m,d}, x_0,x_1,\dots,x_{p^m}, x) n(x,x_\infty) \\
&
= \begin{cases} \pm n_A(\Delta_{m,d-1},x_0,x_1,\dots,x_{p^m},x_\infty) \pm
\text{cyclic permutations of $(x_1,\dots,x_{p^m})$}
& \text{$d$ even}, \\
\mp n_A(\Delta_{m,d-1},x_0,x_{p^m},x_1,\dots,x_{p^m-1},x_\infty)
\pm n_A(\Delta_{m,d-1},x_0,x_1,\dots,x_{p^m},x_\infty)  & \text{$d$ odd.} \end{cases}
\end{aligned} 
\end{equation}
We have omitted the signs, among which are the Koszul signs that occur when cyclically permuting entries; but have indicated one extra degree-independent sign difference in the last case, which comes from \eqref{eq:cellular-differential}.
\begin{figure}
\begin{centering}
\includegraphics{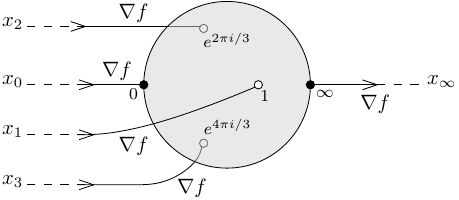}
\caption{\label{fig:half-trajectories}Schematic picture of the constraints on marked points from \eqref{eq:parametrized-maps}.}
\end{centering}
\end{figure}%

The algebraic encapsulation of this goes as follows. Take the Morse homology and cohomology chain complexes, $C_*(f)$ and $C^*(f)$ (with $R$-coefficients; both cohomologically graded, so $C_*(f)$ is concentrated in nonpositive degrees). For any $d$, define a map of degree $-d-2\int_A c_1(M)$,
\begin{equation} \label{eq:chain-quantum-sigma-2}
\begin{aligned}
& C^*(f) \longrightarrow C_*(f)^{\otimes p^m} \otimes C^*(f),
\\
& 
x_0 \longmapsto \sum_{x_1,\dots,x_{p^m},x_\infty} n_A(\Delta_d,x_0,x_1,\dots,x_{p^m},x_\infty) (x_1 \otimes \cdots \otimes x_{p^m}) \otimes x_\infty.
\end{aligned}
\end{equation}
Consider the cyclic permutation action of $\Gamma_m$ on $C_*(f)^{\otimes p^m}$. Here, the convention is that the generator of $\sigma_m$ should move factors to the left, $\sigma_m(x_1 \otimes \cdots \otimes x_{p^m}) = \pm (x_2 \otimes \cdots \otimes x_{p^m} \otimes x_1)$. The relations \eqref{eq:parametrized-numbers} say that if we add up over all $d$, thinking of \eqref{eq:chain-quantum-sigma-2} as the component of $C^d(B\Gamma_m;\cdot)$, the outcome is a chain map of degree $-2\int_A c_1(M)$,
\begin{equation} \label{eq:chain-quantum-sigma-2b}
C^*(f) \longrightarrow C^*(B\Gamma_m;C_*(f)^{\otimes p^m}) \otimes C^*(f).
\end{equation}

\begin{lemma} \label{th:gromov}
(i) Up to chain homotopy, \eqref{eq:chain-quantum-sigma-2b} is independent of all choices.

(ii) It is nullhomotopic unless $A = 0$ or $\int_A c_1(M)>0$.

(iii) There are only finitely many $A$ for which \eqref{eq:chain-quantum-sigma-2b} is nonzero.
\end{lemma}

\begin{proof}
(i) This is a standard argument using an extra parameter $r \in [0,1]$. 

(ii) Assume that $A$ does not satisfy our conditions. Adapt the previous parametrized setup so that for $r = 0$ the inhomogeneous terms $\nu$ are zero, in which case we're looking at straight pseudo-holomorphic maps; for energy reasons, this means that the $r = 0$ stratum is empty, giving a nullhomotopy.

(iii) Gromov compactness tells us that, given upper bounds on $d$ and on $\int_A c_1(A)$, there are only finitely many $A$ for which the numbers \eqref{eq:counting} can be nonzero. On the other hand, for degree reasons, these numbers must be zero if $d$ or $c_1(A)$ are large.
\end{proof}

All other forms of the quantum Steenrod operations are derived from \eqref{eq:chain-quantum-sigma-2b} by purely algebraic manipulations. Take $p^m$ tensor copies of the canonical pairing between Morse chains and cochains, $C^*(f)^{\otimes p^m} \otimes C_*(f)^{\otimes p^m} \longrightarrow R$. From that, the product \eqref{eq:chain-cup-product-with-coefficients}, and the functoriality of $C^*(B\Gamma_m;\cdot)$, one gets a map
\begin{equation} \label{eq:equivariant-pairing-map}
C^*(B\Gamma_m; C^*(f)^{\otimes p^m}) \otimes C^*(B\Gamma_m; C_*(f)^{\otimes p^m}) \longrightarrow C^*(B\Gamma_m).
\end{equation}
Take
\begin{equation} \label{eq:chain-quantum-sigma-3}
\xymatrix{
C^*(B\Gamma_m; C^*(f)^{\otimes p^m}) \otimes C^*(f) \ar[d]^-{\mathit{id} \otimes \eqref{eq:chain-quantum-sigma-2b}} 
\\
C^*(B\Gamma_m; C^*(f)^{\otimes p^m}) \otimes
C^*(B\Gamma_m; C_*(f)^{\otimes p^m})  \otimes C^*(f) \ar[d]^-{\eqref{eq:equivariant-pairing-map} \otimes \mathit{id}} \\ 
C^*(B\Gamma_m) \otimes C^*(f).
}
\end{equation}
Explicit formulae for \eqref{eq:chain-quantum-sigma-3} can be derived from \eqref{eq:chain-cup-product-with-coefficients}. On $C^{\mathrm{even}}(B\Gamma_m;C^*(f)^{\otimes p^m})$, one gets
%
\begin{equation} \label{eq:explicit-qst}
\begin{aligned}
&
t^i(x_1 \otimes \cdots \otimes x_{p^m}) \otimes x_0 \longmapsto  \sum_{x_\infty}
\Big( 
\sum_{d \text{ even}} \pm t^{i+d/2} \, n_A(\Delta_{m,d}, x_0,\dots,x_\infty) 
\\ & \qquad \qquad \qquad \qquad \qquad \qquad \quad + 
\sum_{d \text{ odd}} \pm t^{i+(d-1)/2}\theta \, n_A(\Delta_{m,d}, x_0,\dots,x_{\infty})
\Big) x_\infty
\end{aligned} 
\end{equation}
The formula for the other half of \eqref{eq:chain-quantum-sigma-3}, which we will not need here, generalizes \cite[Equation (4.15)]{seidel-wilkins21} (introduced there without its motivation through cup products). On cohomology, writing things topologically, we get an $H^*_{\Gamma_m}(\mathit{point})$-module map
\begin{equation} \label{eq:quantum-ma}
Q\Sigma_{m,A}: H^*_{\Gamma_m}(M^{p^m}) \otimes H^*(M)
\longrightarrow H^*_{\Gamma_m}(M \times M^{p^m})
\longrightarrow H^*_{\Gamma_m}(M),
\end{equation}
where the $\Gamma_m$-action on the right hand side is trivial. Let's add up over all $A$, with the usual $q^{\int_A c_1(M)}$ weights, to get a single map
\begin{equation} \label{eq:quantum-m}
Q\Sigma_m: H^*_{\Gamma_m}(M^{p^m}) \otimes H^*(M)
\longrightarrow H^*_{\Gamma_m}(M \times M^{p^m})
\longrightarrow H^*_{\Gamma_m}(M)[q];
\end{equation}
this uses Lemma \ref{th:gromov}(ii) to exclude negative powers of $q$. 

Suppose we have an even degree cocycle in $C^*(f)[q]$. One can insert its $p^m$-fold tensor product into the $q$-linearly extended version of \eqref{eq:chain-quantum-sigma-3} and gets a chain map $C^*(f)[q] \longrightarrow C^*(B\Gamma_m) \otimes C^*(f)[q]$. By Lemma \ref{th:diagonal}, the chain homotopy class of this map is depends only on the cohomology class of our original cocycle. As a final step, we use the no-torsion assumption from \eqref{eq:p-adic-number-ring} and K{\"u}nneth-split the right hand side. The outcome, for each $b \in H^{\mathrm{even}}(M)[q]$, is a map of degree $p^m|b|$,
\begin{equation} \label{eq:quantum-mb}
Q\Sigma_{m,b}: H^*(M)[q] \longrightarrow  H^*_{\Gamma_m}(\mathit{point}) \otimes H^*(M)[q].
\end{equation}
As already mentioned in the introduction, one usually extends that further to an $H^*_{\Gamma_m}(\mathit{point})$-linear endomorphism. We will need some basic properties of these operations. The first two are straightforward from the definition:

\begin{lemma} \label{th:q-frobenius}
$Q\Sigma_{m,qb} = q^{p^m} Q\Sigma_{m,b}$.
\end{lemma}

\begin{lemma} \label{th:non-equivariant}
Forgetting the equivariant structure yields a diagram
\begin{equation}
\xymatrix{
H^*(M)[q] \ar@{=}[d] \ar[rr]^-{Q\Sigma_{m,b}} &&
H^*_{\Gamma_m}(\mathit{point}) \otimes H^*(M)[q]  \ar[d]^-{\text{restrict to $H^0_{\Gamma_m}$}} \\
H^*(M)[q] \ar[rr]^-{(b^{\ast_q p^m}) \ast_q \cdot} && H^*(M)[q]
}
\end{equation}
where the bottom $\rightarrow$ is quantum multiplication with the $p^m$-fold quantum power of $b$.
\end{lemma}

\begin{remark} \label{th:not-additive}
From Lemma \ref{th:non-equivariant}, it is clear that \eqref{eq:quantum-m} is not additive in $b$ for $m>1$ (not even modulo $p^m$, so multiplying with powers of $t$ won't help). 
\end{remark}

The next follows from a standard forget-marked-points dimensional argument:

\begin{lemma} \label{th:unitality}
$Q\Sigma_{m,1} = \mathit{id}$.
\end{lemma}

The analogue of \cite[Proposition 4.8]{seidel-wilkins21}, with exactly the same proof, is the following:

\begin{lemma} \label{th:composition}
For $b_1,b_2 \in H^{\mathrm{even}}(M)[q]$, 
\begin{equation} \label{eq:circ}
Q\Sigma_{m,b_1} \circ Q\Sigma_{m,b_2} = Q\Sigma_{b_1 \ast_q b_2}.
\end{equation}
\end{lemma}


\begin{remark}
Another obvious question is to determine the classical ($A = 0$) contribution to $Q\Sigma_{m,b}$. In the $m = 1$ case, this classical part is the cup product with the total Steenrod power of $b$, which we write here as $\mathit{St}_1(b)$ (\cite[Lemma 4.6]{seidel-wilkins21}; see \cite{seidel23} for sign conventions). For general $m$, the same argument shows that the classical contribution is again the cup product with some, as yet undetermined, class $\mathit{St}_m(b) = b^{p^m} + O(t,\theta_m)$. In Example \ref{th:blowup} one finds that for $m>1$,
\begin{equation}
\mathit{St}_m(e) = c_m t^{p^m-1}e, \quad
R/p^{m+1} \ni c_m = (p^m-1)! = \begin{cases} 2 & (p,m) = (2,2) \\
0 & \text{$p>2$ or $m>2$.} \end{cases}
\end{equation}
\end{remark}

\subsection{Covariant constancy}
The version of \cite[Theorem 1.4]{seidel-wilkins21} in our circumstances is:

\begin{proposition} \label{th:covariantly-constant}
For any $m \geq 1$ and $b \in H^*(M)[q]$, the ($H^*_{\Gamma_m}(\mathit{point})$-linear extension of the) operation \eqref{eq:quantum-mb} satisfies
\begin{equation} \label{eq:derivative-of-qsigma}
\nabla_{tq\partial_q} \circ Q\Sigma_{m,b} - Q\Sigma_{m,b} \circ \nabla_{tq\partial_q} = 0.
\end{equation}
\end{proposition}

There are two contributions to the left hand side of \eqref{eq:derivative-of-qsigma}. The first one is obtained by differentiating $b^{\otimes p^m}$ (this was unnecessary in \cite{seidel-wilkins21}, where one could start with a $q$-independent $b$ and then obtain the general case by additivity; which we cannot do for $m>1$, see Remark \ref{th:not-additive}). At the chain level, this contribution is obtained from
\begin{equation} \label{eq:derivative-of-b}
\begin{aligned}
& C^*(f) \xrightarrow{t q\partial_q(b^{\otimes p^m}) \otimes \mathit{id}}
C^*(B\Gamma_m;C^*(f)[q]^{\otimes p^m}) \otimes C^*(f)
\\ & \qquad \qquad \qquad \qquad \qquad \qquad
\xrightarrow{\eqref{eq:chain-quantum-sigma-3}}
C^*(B\Gamma_m) \otimes C^*(f)[q]
\end{aligned}
\end{equation}
Since $b$ is a cocycle in $C^*(f)[q]$, so is $q\partial_q b$. By definition of the differential \eqref{eq:v-coefficients}, one has
\begin{equation}
tq\partial_q(b^{\otimes p^m}) = t\big(q\partial_q b \otimes b \otimes \cdots \otimes b +
\text{cyclic permutations}\big) = d\, \theta(q\partial_qb \otimes b \otimes \cdots \otimes b),
\end{equation}
which shows that \eqref{eq:derivative-of-b} is nullhomotopic.

The second contribution comes from differentiating $Q\Sigma_m$ and then inserting $[b^{\otimes p^m}]$ into that derivative. This strictly follows the corresponding part of \cite{seidel-wilkins21}. By definition,
\begin{equation}
q\partial_q (Q\Sigma_m) = \sum_A ({\textstyle \int_A c_1(M)}) q^{\int_A c_1(M)} Q\Sigma_{m,A}.
\end{equation} 
One interprets this weighted sum geometrically, as coming from a modified version of our operations (this is a form of the divisor axiom for Gromov-Witten invariants). For that, we introduce an additional marked point $z_* \in S$ which can move around freely, with incidence condition
\begin{equation} \label{eq:extra-incidence}
u(z_*) \in D,
\end{equation}
where $D \subset M$ is a submanifold Poincar{\'e} dual to $c_1(M)$ . The outcome is an analogue of \eqref{eq:chain-quantum-sigma-2b}, taking the form of a chain map of degree $2-2\int_A c_1(M)$,
\begin{equation} \label{eq:pi-operation-1}
C^*(f) \longrightarrow C^*\big(B\Gamma_m;C^*(S) \otimes C_*(f)^{\otimes p^m}\big) \otimes C^*(f).
\end{equation}
Here, one uses a $\Gamma_m$-invariant cell decomposition of $S$ \cite[Section 2c]{seidel-wilkins21} to define $C_*(S)$ and its dual $C^*(S)$. This decomposition contains the fixed points $0,\infty$ as invariant $0$-cells, and also a $\Gamma_m$-invariant fundamental chain $[S] \in C_2(S)$. To define \eqref{eq:pi-operation-1}, one counts solutions of \eqref{eq:parametrized-maps} satisfying \eqref{eq:extra-incidence}, where $z_*$ is constrained to lie in one of the cells (if $z_* = 0$ that means bubbling off of an extra sphere with two marked points; and similarly for $z_* = \infty$); see \cite[Sections 4c and 5a]{seidel-wilkins21}. The properties relating this to quantum Steenrod operations are the following direct analogues of \cite[Lemma 5.2--Lemma 5.4]{seidel-wilkins21}:

\begin{proposition} \label{th:extra-point}
Up to chain homotopy, the following holds.

(i) Pairing \eqref{eq:pi-operation-1} with $[S] \in C_2(S)^{\Gamma_m}$ recovers $\int_A c_1(M)$ times \eqref{eq:chain-quantum-sigma-2b}.

(ii) Pairing \eqref{eq:pi-operation-1} with $[0] \in C_0(S)^{\Gamma_m}$ yields the sum, over all $A = A_1+A_2$, of the compositions
\begin{equation}
C^*(f) \xrightarrow{c_1(M) \ast_{A_1} \cdot} C^*(f) \xrightarrow{\text{\eqref{eq:chain-quantum-sigma-2b} for } A_2} C^*(B\Gamma_m;C_*(f)^{\otimes p^m}) \otimes C^*(f) 
\end{equation}
Here, $c_1(M) \ast_{A_1}$ stands for the chain map underlying the $A_1$-component of the quantum product.

(iii) Similarly, pairing with $[\infty] \in C_0(S)^{\Gamma_m}$ yields the sum of
\begin{equation}
\begin{aligned}
& C^*(f) \xrightarrow{\text{\eqref{eq:chain-quantum-sigma-2b} for } A_1} C^*(B\Gamma_m;C^*(f)^{\otimes p^m}) \otimes C^*(f) 
\\ & \qquad \qquad \qquad \xrightarrow{\mathit{id} \otimes (c_1(M) \ast_{A_2} \cdot)}
C^*(B\Gamma_m;C_*(f)^{\otimes p^m}) \otimes C^*(f).
\end{aligned}
\end{equation}
\end{proposition}

The rest of the argument is algebraic manipulation. As in \eqref{eq:chain-quantum-sigma-3}, one can use cup products and pairings to rearrange \eqref{eq:pi-operation-1} into the form
\begin{equation} \label{eq:pi-operation-2}
C^*(f) \otimes C^*(B\Gamma_m; C^*(f)^{\otimes p^m}) \longrightarrow
C^*(B\Gamma_m; C^*(S)) \otimes C^*(f).
\end{equation}
On cohomology, this yields an $H^*_{\Gamma_m}(\mathit{point})$-linear map
\begin{equation} \label{eq:pi-operation-3}
H^*(M) \otimes H^*_{\Gamma_m}(M^{p^m}) \longrightarrow  H^*_{\Gamma_m}(S) \otimes H^*(M)
\end{equation}
where on the target, we use the no-torsion assumption from \eqref{eq:p-adic-number-ring} for simplicity. Proposition \ref{th:extra-point} translates into:

\begin{corollary}
(i) The composition
\begin{equation} \label{eq:integral-over-s}
H^*(M) \otimes H^*_{\Gamma_m}(M^{p^m}) \xrightarrow{\eqref{eq:pi-operation-3}} H^*_{\Gamma_m}(S) \otimes H^*(M) \xrightarrow{\int_S \otimes \mathit{id}} H^*_{\Gamma_m}(\mathit{point}) \otimes H^*(M)
\end{equation}
equals $(\int_A c_1(M)) \, Q\Sigma_{m,A}$.

(ii) The composition
\begin{equation} 
H^*(M) \otimes H^*_{\Gamma_m}(M^{p^m}) \xrightarrow{\eqref{eq:pi-operation-3}} 
H^*_{\Gamma_m}(S) \otimes H^*(M) \xrightarrow{\rho_0 \otimes \mathit{id}} H^*_{\Gamma_m}(\mathit{point}) \otimes H^*(M)
\end{equation}
equals the sum over all $A = A_1+A_2$ of
\begin{equation} \label{eq:restrict-to-0}
H^*(M) \otimes H^*_{\Gamma_m}(M^{p^m}) \xrightarrow{(c_1(M) \ast_{A_1} \cdot) \otimes \mathit{id}} H^*(M) \otimes H^*_{\Gamma_m}(M^{p^m}) \xrightarrow{Q\Sigma_{m,A_2}} H^*_{\Gamma_m}(M).
\end{equation}

(iii) The composition
\begin{equation}
H^*(M) \otimes H^*_{\Gamma_m}(M^{p^m}) \xrightarrow{\eqref{eq:pi-operation-3}} 
H^*_{\Gamma_m}(S) \otimes H^*(M) \xrightarrow{\rho_\infty \otimes \mathit{id}} H^*_{\Gamma_m}(\mathit{point}) \otimes H^*(M)
\end{equation}
equals the sum over all $A = A_1+A_2$ of
\begin{equation} \label{eq:restrict-to-infty}
\begin{aligned} &
H^*(M) \otimes H^*_{\Gamma_m}(M^{p^m}) \xrightarrow{Q\Sigma_{m,A_1}}
H^*_{\Gamma_m}(\mathit{point}) \otimes H^*(M) 
\\ & \qquad \qquad \qquad
\xrightarrow{\mathit{id} \otimes (c_1(M) \ast_{A_1} \cdot)}  H^*_{\Gamma_m}(\mathit{point}) \otimes H^*(M).
\end{aligned}
\end{equation}
\end{corollary}

Applying Lemma \ref{th:localisation}, and adding over all homology classes $A$, yields
\begin{equation}
tq\partial_q(Q\Sigma_m) + (\mathit{id} \otimes c_1(M) \ast_q \cdot) \circ Q\Sigma_m - Q\Sigma_m \circ
((c_1(M) \ast_q \cdot) \otimes \mathit{id}) = 0.
\end{equation}
After inserting the class $[b^{\otimes p^m}] \in H^*_{\Gamma_m}(M^{p^m})$, this is just the relation needed to complete the proof of Proposition \ref{th:covariantly-constant}.

\subsection{Increasing $m$}
The missing piece is the following statement, which generalizes \eqref{eq:gamma12}:

\begin{proposition} \label{th:reducing-diagram}
The operations \eqref{eq:quantum-m} fit into a commutative diagram
\begin{equation} \label{eq:reducing-diagram}
\xymatrix{
H^*(M) \ar[rr]^-{Q\Sigma_{m,b^{\ast_q p}}} && 
H^*_{\Gamma_m}(M)[q] 
\\
H^*(M) \ar[rr]^-{Q\Sigma_{m+1,b}} \ar@{=}[u] && 
H^*_{\Gamma_{m+1}}(M)[q] \ar[u]_-{\text{\rm restrict to $\Gamma_m \subset \Gamma_{m+1}$}}
}
\end{equation}
\end{proposition}

Consider
\begin{equation} \label{eq:operation-with-reduced-m}
C^*(f) \xrightarrow{\eqref{eq:chain-quantum-sigma-2b}} C^*\big(B\Gamma_{m+1};C_*(f)^{\otimes p^{m+1}}\big) \otimes C^*(f)
\xrightarrow{\eqref{eq:induced-map-with-coefficients}} C^*(B\Gamma_m;C^*(f)^{\otimes p^{m+1}}) \otimes C^*(f).
\end{equation}
Our cell decompositions have the property that
\begin{equation}
\Delta_{m,d} = \begin{cases} \Delta_{m+1,d} & \text{if $d$ is even,} \\
\Delta_{m+1,d} \cup \sigma_{m+1}(\Delta_{m+1,d}) \cup \cdots
\cup \sigma_{m+1}^{p-1}(\Delta_{m+1,d}) & \text{if $d$ is odd.}
\end{cases}
\end{equation}
In the second case, the only overlaps happen at the boundary. Hence, applying \eqref{eq:induced-map-with-coefficients} amounts to considering operations parametrized by $\Delta_{m,d}$ rather than $\Delta_{m,d+1}$. In particular, the composition \eqref{eq:operation-with-reduced-m} is the operation associated to the sphere $S$ with $p^{m+1}+2$ marked points, but using only equivariance with respect to the subgroup $\Gamma_m$. This gives us additional freedom: we can deform the construction by replacing the $\zeta_{m+1,j}$ in \eqref{eq:parametrized-maps} with any collection of pairwise distinct points
\begin{equation} \label{eq:zeta-symmetry}
\zeta_{m+1,1},\dots,\zeta_{m+1,p^{m+1}} \in S \setminus \{0,\infty\}, \;\; e^{2\pi i/p^m}\zeta_{m+1,j} = 
\begin{cases} \zeta_{m+1,j+p} & j+p \leq p^{m+1}, \\ \zeta_{m+1,j+p-p^{m+1}} & \text{otherwise.} \end{cases}
\end{equation}
Specifically, what we want to do is to move the points $(\zeta_{m+1,1},\zeta_{m+1,2},\dots,\zeta_{m+1,p})$ so that they come together at $\zeta_{m,1}$ at roughly the same rate; by \eqref{eq:zeta-symmetry}, this means that all the $p^{m+1}$ points come together in groups of $p$ each. In the limit as the points collide and bubble off, one gets (excluding further bubbling and breaking, which are positive codimension phenomena) pseudo-holomorphic map configurations
\begin{equation} \label{eq:configurations}
\left\{
\begin{aligned}
& u: S \longrightarrow M, \\
& \tilde{u}_1,\dots,\tilde{u}_{p^m}: \tilde{S} \longrightarrow M, \\
& u(0) \in W^u(x_0), \; u(\infty) \in W^s(x_\infty), \\
& u(\zeta_{m,j}) = \tilde{u}_j(\infty) \text{ for $j = 1,\dots,p^m$}, \\
& \tilde{u}_1(\tilde{\zeta}_1) \in W^u(x_1), \dots, \tilde{u}(\tilde{\zeta}_p) \in W^u(x_p), \\
& \dots, \\
& \tilde{u}_{p^m}(\tilde{\zeta}_1) \in W^u(x_{p^{m+1}-p^m+1}), \dots,
\tilde{u}_{p^m}(\tilde{\zeta}_p) \in W^u(x_{p^{m+1}}).
\end{aligned}
\right.
\end{equation}
Here, $\tilde{S}$ is a sphere, which carries marked points $\tilde{\zeta}_1,\dots,\tilde{\zeta}_p \neq \infty$. Those marked points are in fixed position (which position one gets depends on how one moves the original points together). Of course, the entire situation is still parametrized by $w \in S^\infty$ and subject to $\Gamma_m$-equivariance constraints; but one can achieve, without affecting transversality, that the inhomogeneous term on $\tilde{S}$ is independent of $w$. 

Concretely, the outcome of this degeneration is a chain homotopy relating \eqref{eq:operation-with-reduced-m} to a version defined using \eqref{eq:configurations}. At this point, we adopt a standard technique, which is to modify the adjacency condition in the fourth line of \eqref{eq:configurations} by inserting a finite length gradient flow line between the two points. In other words, we define another parametrized moduli space depending on $r \in [0,\infty)$, where the condition is
\begin{equation}
u(\zeta_{m,j}) = \phi^r(\tilde{u}_j(\infty)),
\end{equation}
$\phi^r$ being the gradient flow of $f$. In the limit $r \rightarrow \infty$, the equation on the $\tilde{S}$-components completely decouples from that on $S$ (having its own $W^s$ and $W^u$ intersection conditions). Since that equation is set up using marked points $\tilde{\zeta}_1,\dots,\tilde{\zeta}_p,\infty$ in fixed position, it defines the $p$-fold quantum product. What the entire geometric construction yields is a chain homotopy commutative diagram
\begin{equation} \label{eq:reduction-diagram}
\xymatrix{
C^*(f) \ar[r]^-{\eqref{eq:chain-quantum-sigma-2b}} &
C^*(B\Gamma_m; C_*(f)^{\otimes p^m}) \otimes C^*(f) \ar[r]
& 
C^*(B\Gamma_m; C_*(f)^{\otimes p^{m+1}}) \otimes C^*(f)
\\
C^*(f) \ar[rr]^-{\eqref{eq:chain-quantum-sigma-2b}} \ar@{=}[u] &&
 C^*(B\Gamma_{m+1};C_*(f)^{\otimes p^{m+1}}) \otimes C^*(f)
 \ar[u]_-{\eqref{eq:induced-map-with-coefficients}}
}
\end{equation}
Here, the unlabeled map is the tensor product of $p^m$ copies of the map $C_*(f) \rightarrow C_*(f)^{\otimes p}$ dual to the $p$-fold quantum product. We have omitted considerations of homology classes of pseudo-holomorphic curves, which would have complicated the diagram somewhat, but they are easy to see: if the bottom line of \eqref{eq:reduction-diagram} counts contributions in class $A$, then the top line should really be a sum, where the second $\rightarrow$ runs over classes $A_1,\dots,A_{p^m}$, and the first $\rightarrow$ is the contribution of $A - A_1 - \cdots - A_{p^m}$. From here, pairing with $b^{\otimes p^{m+1}}$ immediately leads to \eqref{eq:reducing-diagram}.

\subsection{Proof of Theorem \ref{th:p-adic}}
For $b \in H^{\mathrm{even}}(M)[q^{\pm 1}]$ define a map
\begin{equation} \label{eq:quantum-m-endo-2}
Q\Sigma_{m,b}: H^*(M)[q^{\pm 1}] \longrightarrow H^*_{\Gamma_m}(\mathit{point}) \otimes H^*(M)[q^{\pm 1}].
\end{equation}
If $b$ has no negative powers of $q$, this is obtained from \eqref{eq:quantum-mb} by extending coefficients; and one generalizes it by setting $Q\Sigma_{m,q^{-1}b} = q^{-p^m} Q\Sigma_{m,b}$, compatibly with Lemma \ref{th:q-frobenius}. Proposition \ref{th:covariantly-constant} remains true, because $tq\partial_q (q^{-p^m}) = -p^m \,tq^{-p^m} = 0 \in H^*_{\Gamma_m}(\mathit{point})[q^{\pm 1}]$.

Specialize to the case where $b = b^{\ast_q p}$ is (a degree $0$ class and) equals its $p$-th power. Then, Proposition \ref{th:reducing-diagram} relates $Q\Sigma_{m+1,b}$ and $Q\Sigma_{m,b}$, allowing us to define an operation $Q\Sigma_{\infty,b}$, which fits into a diagram
\begin{equation} \label{eq:reducing-diagram-2}
\xymatrix{
H^*(M)[q^{\pm 1}] \ar[rr]^-{Q\Sigma_{m,b}} && 
H^*_{\Gamma_m}(\mathit{point}) \otimes H^*(M)[q^{\pm 1}]
\\
H^*(M)[q^{\pm 1}] \ar[rr]^-{Q\Sigma_{\infty,b}} \ar@{=}[u] && \ar[u] 
H^*(M)[q^{\pm 1}][[t]]
}
\end{equation}
Here we have considered each power of $t$ separately, taken the inverse limit (compare Lemma \ref{th:gamma-infinity}), and then assembled those into $H^*(M)[q^{\pm 1}][[t]]$, which is the completion of $H^*_{\Gamma_\infty}(\mathit{point}) \otimes H^*(M)[q^{\pm 1}]$. This is somewhat coarser than what is actually the case:

\begin{lemma}
The operation $Q\Sigma_{\infty,b}$ takes values in $H^*(M)[q^{\pm 1}]\lla t \rra$.
\end{lemma}

\begin{proof}
Suppose that the nontrivial terms in $b$ come with powers $q^{-\alpha}$ and higher. Then, by definition, $Q\Sigma_{m,b}$ contains only powers $q^{-p^m\alpha}$ or higher. Since $Q\Sigma_{m,b}$ is of degree $0$, this means that any term with a power of $t$ higher than $p^m\alpha + \beta$, $\beta = \mathrm{dim}_{\bC}(M)$, has to be zero. By construction, the same is true for the mod $p^m$ reduction of $Q\Sigma_{\infty,b}$.
\end{proof}

Now suppose we have a collection \eqref{eq:idempotents}. Extend $Q\Sigma_{\infty,e_i}$ to endomorphisms of $H^*(M)[q^{\pm 1}]\lla t \rra$, which are covariantly constant by Proposition \ref{th:covariantly-constant}. From Lemmas \ref{th:unitality} and \ref{th:composition}, it follows that
\begin{equation}
Q\Sigma_{\infty,e_i} \circ Q\Sigma_{\infty,e_j} = \begin{cases} Q\Sigma_{\infty,e_i} & i = j, \\
0 & \text{$i \neq j$.}
\end{cases}
\end{equation}
From Lemma \ref{th:non-equivariant} and \eqref{eq:reducing-diagram-2}, the $t = 0$ reduction of $Q\Sigma_{\infty,e_i}$ is the quantum product with $e_i$.

\appendix
\section{Elementary considerations\label{sec:appendix}}
We collect here some properties of formal power series connections, which are relevant for the discussion in Section \ref{sec:overview}. 
The general context is that of an integral domain $R$, but we will impose additional conditions as needed. We consider connections
\begin{equation} \label{eq:pole}
\nabla_{\tau^2\partial_\tau} = \tau^2\partial_\tau + A^0 + A^1\tau + \cdots
\quad A^m \in \mathit{Mat}_{r \times r}(R).
\end{equation}

\begin{lemma} \label{th:splitting-lemma}
Assume that the leading term $A^0$ of our connection satisfies \eqref{eq:differences-are-invertible}. 
Then there is a unique splitting of $R^r[[\tau]]$, compatible with the connection and which, for $\tau = 0$, reduces to the generalized eigenspace decomposition for $A^0$. Moreover, any covariantly constant endomorphism preserves this decomposition.
\end{lemma}

A classical reference is \cite[Chapter IV, Theorem 11.1]{wasow65}. For expositions, see \cite[Section 6.1]{hugtenburg24} or (in our context) \cite[Corollary 2.4]{chen24c}.

\begin{lemma} \label{th:non-resonant}
Take $R$ to be a field of characteristic $0$. Take a connection \eqref{eq:pole} with a simple pole, meaning $A^0 = 0$. In addition we assume that $A^1$ is non-resonant (no two eigenvalues differ by a nonzero integer).

(i) Any covariantly constant endomorphism of the connection over $R((\tau))$ is in fact defined over $R[[\tau]]$.

(ii) The constant term $E^0$ of such an endomorphism satisfies $[A^1,E^0] = 0$.

(iii) Conversely, given any $E^0$ with $[A^1,E^0] = 0$, it can be uniquely extended to an endomorphism of the connection.
\end{lemma}

This is a straightforward order-by-order computation. The following two examples show that the non-resonance condition is necessary both for existence and uniqueness; with a view to our application, we focus on idempotent endomorphisms.
%
%
%
%

\begin{example} \label{th:non-existence}
Take
\begin{equation}
\tau^2\partial_\tau + \begin{pmatrix} -\tau & \tau^2 \\ 0 & 0 \end{pmatrix}.
\end{equation}
Then, $E^0 = \mathrm{diag}(1,0)$ commutes with $A^1 = \mathrm{diag}(-1,0)$, but does not admit a covariantly constant extension. This can be checked by hand; more conceptually, it is a consequence of the fact that our connection has non-semisimple monodromy.
\end{example}

\begin{example} \label{th:non-uniqueness}
Take $
\tau^2\partial_\tau + \mathrm{diag}(0,0,0,-\tau)$
(this is the trivial connection after a monomial base-change). Then, 
\begin{equation}
\textstyle \mathrm{diag}(1,1,0,0) \;\; \text{ and } \;\; \left(\begin{smallmatrix} 1 &&& \\ & 1 && \\ && 0 & \\ \tau &&& 0 \end{smallmatrix}\right)
\end{equation}
are both covariantly constant endomorphisms; idempotent; and have the same constant term.
\end{example}

\begin{lemma} \label{th:regular}
Let $R$ be an algebraically closed field of characteristic $0$. For any connection $\nabla$ over $R((\tau))$, the covariantly constant endomorphisms form a finite-dimensional $R$-vector space.
\end{lemma}

The case of a connection with a regular singularity follows from Lemma \ref{th:non-resonant}. The general case can be reduced to this using the Hukuhura-Turrittin-Levelt decomposition (see e.g.\ \cite{babbitt-varadarajan83}).
%

\begin{proposition} \label{th:constructible-2}
Let $R$ be an algebraically closed field of characteristic $0$. Take a connection \eqref{eq:pole} and projections $E_i^0 \in \mathit{Mat}_{r \times r}(R)$,
\begin{equation}
[E_i^0,A^0] = 0, \quad 1 = E_1^0 + \cdots E_m^0, \quad E_i^0E_j^0 = \begin{cases}
E_i^0 & i = j, \\ 0 & i \neq j. \end{cases} 
\end{equation}
Suppose that, over $\tilde{R}[[\tau]]$ for some larger field $\tilde{R} \supset R$, there is a splitting of our connection which for $\tau = 0$ reduces to that given by the projections. Then, such a splitting already exists over $R$.
\end{proposition}

\begin{proof}
We are looking for covariantly constant endomorphisms 
\begin{equation} \label{eq:projections}
E_i = E_i^0 + O(\tau), \quad I = E_1 + \cdots + E_m, \quad E_iE_j = \begin{cases} E_i & i = j, \\ 0 & i \neq j. \end{cases}
\end{equation}
Within the finite-dimensional vector space of all endomorphisms over $R((\tau))$, this is a variety defined by equations that are linear (vanishing of the coefficients with negative powers of $\tau$, and the first two parts of the equation above) and quadratic (last part). Generally, if an affine variety defined over $R$ has an $\tilde{R}$-point, it also has an $R$-point (Lefschetz principle, see e.g.\ \cite[Corollary 2.2.10]{marker02}).
\end{proof}


\end{document}